\documentclass[12pt]{amsart}

\usepackage[english]{babel}
\usepackage[utf8x]{inputenc}
\usepackage[T1]{fontenc}
\usepackage{cite}
\usepackage{color}
\usepackage{xypic}

\usepackage[a4paper,top=3cm,bottom=2cm,left=3cm,right=3cm,marginparwidth=1.75cm]{geometry}

\usepackage{amsmath, amsthm, amssymb}
\usepackage{graphicx}
\usepackage[colorinlistoftodos]{todonotes}
\usepackage[colorlinks=true, allcolors=blue]{hyperref}
\usepackage{skull}
\newtheorem{theorem}{Theorem}

\newtheorem{prop}{Proposition}

\newcommand{\PP}{\mathbb{P}}

\usepackage{subcaption}

\newtheorem*{remark*}{Remark}
\newtheorem*{example*}{Example}
\newtheorem*{defn*}{Definition}

\newtheorem*{from1*}{Proposition 2 from~\cite{dbm}}

\newtheorem*{from2*}{Proposition 3 from~\cite{dbm}}

\title{Dessins d'enfants for single-cycle Belyi maps}
\author{Michelle Manes, Gabrielle Melamed, Bella Tobin}
\thanks{MM partially supported by  the Simons Foundation grant \#359721.}

\email{mmanes@math.hawaii.edu, gmelamed@hawaii.edu, tobin@math.hawaii.edu}

\begin{document}
\maketitle

\begin{abstract}
 Riemann's Existence Theorem gives  the following bijections: 
\begin{enumerate}
\item Isomorphism classes of {\em Belyi maps} of degree $d$.
\item Equivalence classes of  {\em generating systems} of degree $d$. 
\item Isomorphism classes of {\em dessins d'enfants} with $d$ edges.
\end{enumerate}

In previous work, the first author and collaborators exploited the correspondence between Belyi maps and their generating systems to provide explicit equations for two infinite families of dynamical Belyi maps. 
We complete this     picture by describing the dessins d'enfants for these two families.
\end{abstract}

\section{Introduction}

Let $X$ be a smooth projective curve. A Belyi map $f:X\rightarrow \PP^1$ is a finite cover that is ramified only over the points $0,1,$ and $\infty$. The genus of the Belyi map is the genus of the covering curve $X$. 
There are multiple ways to realize a Belyi map of degree~$d$ (see, for example,~\cite{RETref, DessinIntro, WhatIs}):
\begin{enumerate}
\item
explicitly, as a degree $d$ function between projective curves;
\item
combinatorially, as a generating system of degree $d$; and
\item
topologically, as a dessin d'enfant with $d$ edges.
\end{enumerate} 

A \emph{generating system} of degree $d$ is a triple of permutations   $(\sigma_0,\sigma_1, \sigma_\infty) \in S_d^3$ with the property that $\sigma_0\sigma_1 \sigma_\infty =1$ and the subgroup $\langle \sigma_0,\sigma_1 \rangle \subseteq S_d$ is transitive. 
%Two generating systems  $(\sigma_0,\sigma_1, \sigma_\infty)$ and  $(\sigma_0',\sigma_1', \sigma_\infty')$ are equivalent if there is a $\tau \in S_d$ such that for each $i$, $\tau^{-1}\sigma_i\tau = \sigma_i'$.

A \emph{dessin d'enfant} (henceforth \emph{dessin}) is a connected bipartite graph embedded in an orientable surface. The dessin has a fixed cyclic ordering of its edges at each vertex; this manifests as a labeling.

In general, it is a simple matter to describe a dessin from either a generating system or a function. 
Given $f:X \to \PP^1$: 
\begin{itemize}
\item
Define a black vertex
for each inverse image of 0.
\item
Define a white vertex for each
 inverse image of 1.  
 \item 
 Define edges
by the inverse images of the line
segment $(0, 1)\in \PP^1$.   
\end{itemize}
This process yields a connected bipartite graph. The labeling of the edges arises from the local monodromy around the vertices.

Similarly, given a generating system $(\sigma_0,\sigma_1, \sigma_\infty) \in S_d^3$, we create a dessin with edges labeled $\{1, 2, . . . , d\}$ via the following recipe:
\begin{itemize}\label{recipe:GenToDessin}
\item
Draw a black vertex for each cycle in $\sigma_0$ (including the one-cycles). The cycles in $\sigma_0$ then give an ordering of edges around each of these vertices.
\item
Draw a a white vertex for each cycle in $\sigma_1$ (including the one-cycles). The cycles in $\sigma_1$ then give an ordering of edges around each of these vertices.
\end{itemize}
 This determines a bipartite graph. Since $\sigma_0$ and $\sigma_1$ generate a transitive subgroup of $S_d$, the graph is connected.

It is equally straightforward to describe a generating system from a dessin. The difficulty in completing the picture is often in giving an explicit function realizing the Belyi map as a covering of $\PP^1$.
 For some recent results in this area, see~\cite{SV, Zvonkin}.

In some simple cases, however, we can explicitly realize this triple correspondence for an infinite family of Belyi maps. For example, for pure power maps we have the following:

\begin{description}
\item[Belyi map]
\begin{align*}
f: \PP^1 &\to \PP^1\\
z &\mapsto z^d
\end{align*}

\item[Generating system]
\[ 
\sigma_0 =  d\text{-cycle}, \qquad \sigma_1 = \text{trivial}, \qquad \sigma_\infty = d\text{-cycle}.
\]

\item[Dessin d'enfant] \ 

% \begin{center}
% \includegraphics[scale=.5]{powerdessin}
% \end{center}

\begin{center}
\begin{figure}[h!]
\includegraphics{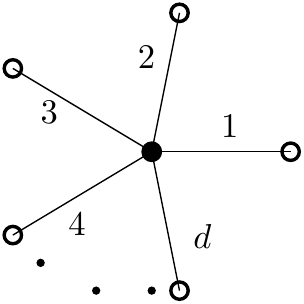}
\caption{The dessin for the degree $d$ power map.}
\end{figure}
\end{center}

\end{description}

 And for Chebyshev polynomials we have:

\begin{description}
\item[Belyi map]
\begin{align*}
f: \PP^1 &\to \PP^1\\
z &\mapsto T_d(z)
\end{align*}

\item[Generating system]
\begin{align*}
\sigma_0 &= (23)(45)\cdots((d-1)d)
&&\text{or } (23)(45)\cdots((d-2)(d-1)).\\
\sigma_1 &=  (12)(34)\cdots((d-2)(d-1))
&&\text{or } (12)(34)\cdots((d-1)d).\\
\sigma_\infty &= d\text{-cycle}.
\end{align*}

\item[Dessin d'enfant] \ 

% \begin{center}
% \includegraphics[scale=.5]{chebydessin}
% \end{center}

\begin{center}
\begin{figure}[h!]
\includegraphics{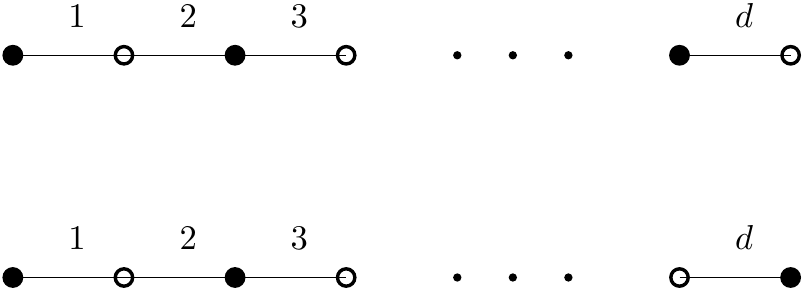}
\caption{The Chebyshev dessins : odd $d$ (top) and even $d$ (bottom).}
\end{figure}
\end{center}

\end{description}

Motivated by work in arithmetic dynamics,  the authors of~\cite{dbm} study \emph{normalized single-cycle dynamical Belyi maps}. 
 The authors begin with a generating system, and they are able to give explicit formulas for two new infinite families of Belyi maps.

In this note, we give a simple description of the dessins for genus~$0$ single-cycle dynamical Belyi maps. As an application, we describe the dessins for the two infinite families of maps in~\cite{dbm}, completing the triptych in these cases.

\subsection*{Acknowledgments}
The authors thank Edray Goins for asking the question that prompted this investigation. We also thank the referee for numerous helpful comments.

\section{Dessins for single-cycle Belyi maps}\label{sec:SingCycDyn}
Riemann's Existence Theorem gives  the following bijections: 
\begin{enumerate}
\item \emph{equivalence classes} of   generating systems of degree $d$. 
\item \emph{isomorphism classes} of Belyi maps $f: X \to \PP^1$ of degree $d$.
\item \emph{isomorphism classes} of  dessins d'enfants with $d$ edges.
\end{enumerate}
In the sequel, we will use equivalent generating systems and isomorphic Belyi maps to simplify the exposition.

\begin{defn*}
Let $f:X \to \PP^1$ be a Belyi map. If $f$ has a single ramification point above each branch point, we say that $f$ is a {\bf single-cycle Belyi map}.
\end{defn*}

Following~\cite{dbm}, we require that single-cycle Belyi maps have exactly three critical points.  So all Belyi maps under consideration have degree $d \geq 3$. If $f:\PP^1 \to \PP^1$ has only two critical points, then it is isomorphic to a pure power map. We have already described the dessins for this case in the Introduction.

In the genus~$0$ single-cycle case, the permutations $\sigma_0$, $\sigma_1$, and $\sigma_\infty$ each correspond to a single nontrivial cycle. Hence the generating system can be written as $(e_0, e_1, e_\infty)$ where $e_i$ gives the length of the nontrivial cycle in $\sigma_i$. Equivalently, $e_i$ is the ramification index of the unique critical point above $i \in \PP^1$.

 The following theorem classifies the dessins d'enfants for all genus~$0$ single-cycle Belyi maps.

\begin{theorem}\label{thm:main}
Let $f:\PP^1 \to \PP^1$ be a degree-$d$ single-cycle Belyi map with generating system $(e_0, e_1, e_\infty)$. Then $f$ admits a planar dessin d’enfant with:
\begin{itemize}
\item
$d − e_1$ white vertices of degree one connected to a black vertex of degree $e_0$, 
\item
$d− e_0$ black vertices of degree one connected to a white vertex of degree $e_1$,
and
\item
$e_0 + e_1 − d$ edges connecting the black vertex of degree $e_0$ and the white vertex of degree $e_1$.
\end{itemize}
\end{theorem}
See Figure~\ref{fig:genbelyi}.

\begin{proof}
Let
\[
\sigma_0 = (d-e_0+1, d-e_0+2, \ldots, d)
\text{ and }
\sigma_1 = (1, 2, \ldots, e_1).
\]
It follows from Riemann-Hurwitz that $e_0 + e_1 + e_\infty = 2d+1$, so 
\[
d+1 \leq e_0 + e_1 \leq 2d-1.
\]
Therefore $\langle \sigma_0, \sigma_1 \rangle$ is transitive since $d-e_0+1 \leq e_1$.
 
The result then follows immediately from the recipe for producing dessins from generating systems described in the Introduction.
\end{proof}

% \begin{figure*}[h!]
% \includegraphics[scale=.7]{genbelyi.png}
% \caption{The degree $d$, genus~$0$ single-cycle dessin with combinatorial type $(e_0,e_1,e_\infty)$.}\label{fig:genbelyi}
% \end{figure*}

\begin{figure*}[h!]
\includegraphics[scale = 1]{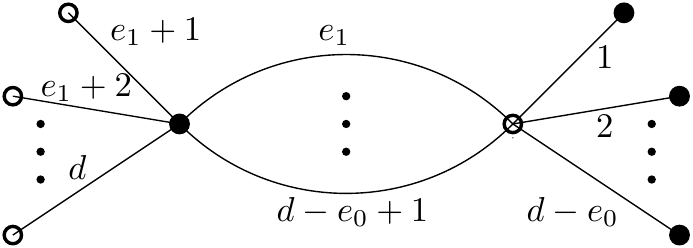}
\caption{The degree $d$, genus~$0$ single-cycle dessin with combinatorial type $(e_0,e_1,e_\infty)$.}\label{fig:genbelyi}
\end{figure*}

Recall that the {\bf diameter} of a graph is the maximal number of vertices traversed in a path. 
The following result gives a restriction on the diameter of the dessins for all single-cycle Belyi maps (not only the genus~$0$ case). 
\begin{prop}
All single-cycle Belyi maps admit a dessin d'enfant of diameter at most~$4$. 
\end{prop}
\begin{proof}
Let $f:X \to \PP^1$ be a single-cycle  Belyi map. So there is a unique ramification point above $0$ and a unique ramification point above $1$. Hence there are exactly two vertices in the dessin with degree greater than $1$. This implies that the longest path can only include those two vertices and two additional vertices, one black and one white. 
\end{proof}

\section{New triptychs for single-cycle Belyi maps}\label{sec:NewTrips}
Applying Theorem~\ref{thm:main} to the Belyi maps in~\cite{dbm} allows us to describe the three-way correspondence for two new infinite families of Belyi maps.

\subsection{Single-cycle polynomials} Let $f: \PP^1 \to \PP^1$ be a degree~$d$ single-cycle Belyi  polynomial, so $(e_0, e_1, e_\infty) = (d-k, k+1, d)$ for some $1 \leq k < d-1$. We have the following correspondence:

\begin{description}
\item[Belyi map] (from~\cite{dbm})
\begin{align*}
f: \PP^1 &\to \PP^1\\
z &\mapsto cx^{d-k} (a_0x^k+\ldots +a_{k-1}x+a_k),
\end{align*}
where
\[
a_i:=  \frac{(-1)^{k-i}}{(d-i)} \binom{k}{i} \hspace{2mm} \text{and} \hspace{2mm} c=\frac{1}{k!}              \prod_{ j=0}^{k}(d-j). \]

\item[Generating system]
\[ 
\sigma_0 =  (d-k)\text{-cycle}, \qquad \sigma_1 = (k+1)\text{-cycle}, \qquad \sigma_\infty = d\text{-cycle}.
\]

\item[Dessin d'enfant] See Figure~\ref{fig:genpoly}.

\begin{figure*}[h!]
\includegraphics[scale=.8]{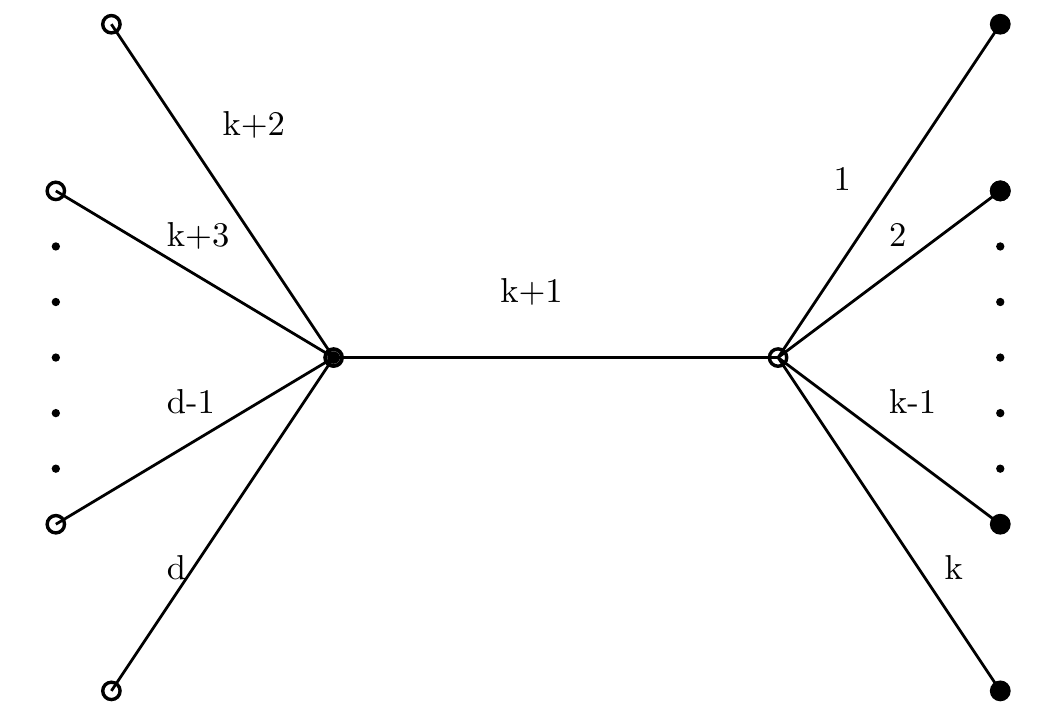}
\caption{The dessin with combinatorial type $(d-k,k+1,d)$.}\label{fig:genpoly}
\end{figure*}

\end{description}

\begin{example*}
The dessin for the polynomial $f(z) = z^3(6z^2-15z+10)$, which has combinatorial type $(3, 3, 5)$, is shown in Figure~$\ref{fig:specificpoly}$.  

\begin{figure*}[h!]
\includegraphics[scale=.9]{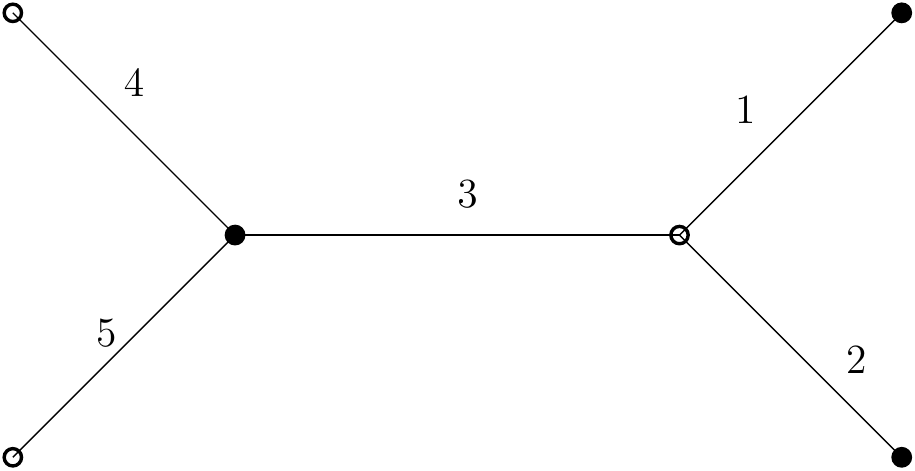}
\caption{The dessin with combinatorial type $(3,3,5)$.}\label{fig:specificpoly}
\end{figure*}

\end{example*}

\subsection{Symmetric single-cycle Belyi maps}
We turn now to the second family of Belyi maps described  in~\cite{dbm}. These have combinatorial type $\left(d-k,2k+1,d-k\right)$, meaning that the critical points above $0$ and~$\infty$ have the same ramification index. (Here $d$ is the degree of the Belyi map and $1 \leq k < d-1$.) We have the following correspondence:

\begin{description}
\item[Belyi map] (from~\cite{dbm})
\begin{align*}
f: \PP^1 &\to \PP^1\\
z &\mapsto x^{d-k}\left( \frac{a_0x^k-a_1x^{k-1}+\ldots
            +(-1)^ka_k}{(-1)^k a_k x^k +\ldots -a_1x+ a_0}\right),
\end{align*}
where
\[
a_i:=\binom{k}{i}\prod_{k+i+1\leq j \leq 2k} (d-j)\prod_{0\leq j \leq i-1}(d-j) =
k!\binom{d}{i}\binom{d-k-i-1}{k-i}.
\]

\item[Generating system]
\[ 
\sigma_0 =  (d-k)\text{-cycle}, \qquad \sigma_1 = (2k+1)\text{-cycle}, \qquad \sigma_\infty = (d-k)\text{-cycle}.
\]

\item[Dessin d'enfant] \ 

\begin{figure*}[h!]
\includegraphics[scale=.75]{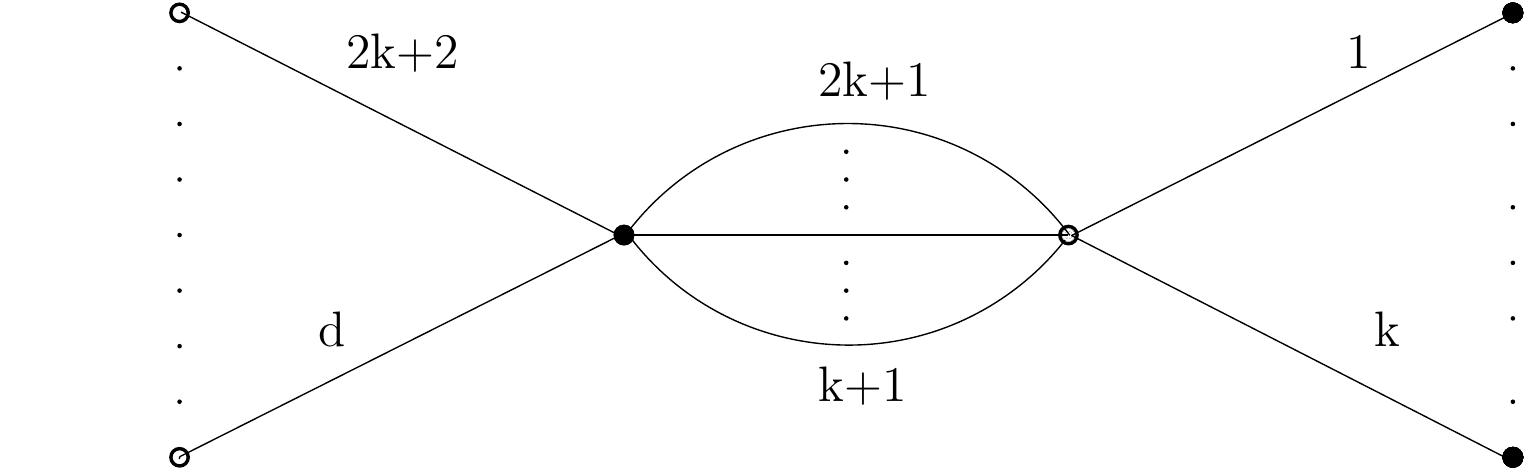}
\caption{The dessin with combinatorial type $(d-k,2k+1,d-k)$.}\label{fig:genrat}
\end{figure*}

\end{description}

\begin{example*}
The dessin for the map 
\[f(z) = z^8 \left(\frac{ 42 z^2-120 z+90}{90z^2-120 z+42}\right),
\]
which has combinatorial type $(8, 5, 8)$, is shown in Figure~$\ref{fig:specificrat}$. 

\begin{figure*}[h!]
\includegraphics[scale=.75]{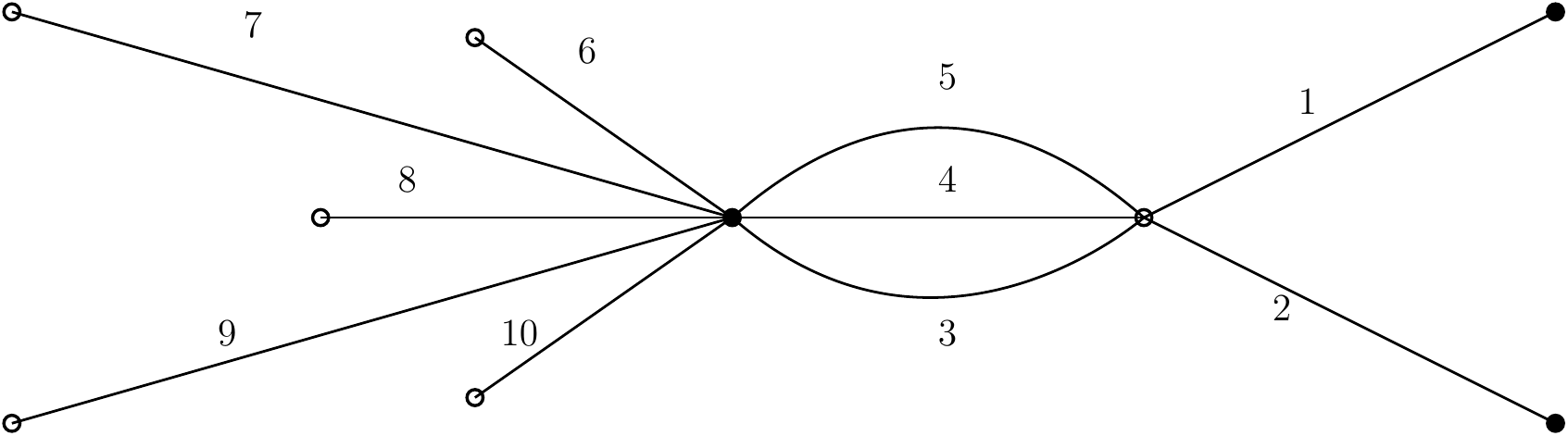}
\caption{The dessin with combinatorial type $(8,5,8)$.}\label{fig:specificrat}
\end{figure*}

\end{example*}

\bibliography{dbm_manes}{}

\begin{thebibliography}{1}

\bibitem{dbm}
Jacqueline Anderson, Irene Bouw, Ozlem Ejder, Neslihan Girgin, Valentijn
  Karemaker, and Michelle Manes.
\newblock Dynamical {B}elyi maps, 2017.

\bibitem{RETref}
Frits Beukers and Hans Montanus.
\newblock Explicit calculation of elliptic fibrations of {$K3$}-surfaces and
  their {B}elyi-maps.
\newblock In {\em Number theory and polynomials}, volume 352 of {\em London
  Math. Soc. Lecture Note Ser.}, pages 33--51. Cambridge Univ. Press,
  Cambridge, 2008.

\bibitem{DessinIntro}
Ernesto Girondo and Gabino Gonz\'alez-Diez.
\newblock {\em Introduction to compact {R}iemann surfaces and dessins
  d'enfants}, volume~79 of {\em London Mathematical Society Student Texts}.
\newblock Cambridge University Press, Cambridge, 2012.

\bibitem{SV}
Jeroen Sijsling and John Voight.
\newblock On computing {B}elyi maps.
\newblock In {\em Num\'ero consacr\'e au trimestre ``{M}\'ethodes
  arithm\'etiques et applications'', automne 2013}, volume 2014/1 of {\em Publ.
  Math. Besan\c{c}on Alg\`ebre Th\'eorie Nr.}, pages 73--131. Presses Univ.
  Franche-Comt\'e, Besan\c{c}on, 2014.

\bibitem{WhatIs}
Leonardo Zapponi.
\newblock What is...a dessin d’enfant?
\newblock {\em Notices of the Amer. Math. Soc.}, 50:788--789, 2003.

\bibitem{Zvonkin}
Alexander Zvonkin.
\newblock Belyi functions: examples, properties, and applications.
\newblock \url{http://www.labri.fr/perso/zvonkin/Research/belyi.pdf}.

\end{thebibliography}
\bibliographystyle{plain}

\end{document}